\documentclass[a4paper,10pt]{amsart}

\usepackage{acatta}

\newcommand{\blankSpace}{\cdot}

\newcommand{\Fix}{\operatorname{Fix}}
\newcommand{\tr}{\operatorname{tr}}
\newcommand{\diag}{\operatorname{diag}}

\renewcommand{\Im}{\operatorname{Im}}
\renewcommand{\Re}{\operatorname{Re}}

\title{Split Nakamura manifolds and their automorphisms}

\author{Andrea Cattaneo}
\address{Universit\`a di Parma\\
Dipartimento di Scienze Matematiche, Fisiche e Informatiche\\
Unit\`a di Matematica e Informatica\\
Parco Area delle Scienze 53/A, 43124, Parma, Italy}
\email{andrea.cattaneo@unipr.it}


\begin{document}

\begin{abstract}
In this paper we study the class of \emph{split Nakamura manifolds}, which are a type of solvmanifolds generalizing Nakamura's threefold, defined as quotients of the semidirect product $\IC^n \rtimes_\rho \IC$ by a lattice. We discuss their de Rham and Dolbeault cohomology, with emphasis on the degeneration of the Fr\"olicher spectral sequence and the $\del\delbar$-Lemma, and their deformations. Finally, we describe their automorphism group in detail.
\end{abstract}

\maketitle

\tableofcontents

\section{Introduction}

The class of compact complex solvmanifolds is particularly rich in manifolds providing behaviors which are very different from the K\"ahler manifolds. In this paper we deal with the class of \emph{Nakamura manifolds}, which were introduced in \cite[$\S$3]{Cattaneo-Tomassini} as a generalization of the completely solvable Nakamura threefold which appeared in \cite[$\S$2, Case III-(3)]{Nakamura}. These manifolds are defined as the quotient of the semidirect product $G = \IC^n \rtimes_\rho \IC$ by the action of (left) translations by the elements of a lattice $\Gamma$. In particular, we will consider mainly those Nakamura manifolds which arise from a lattice whose structure is compatible with the structure of semidirect product of $G$, which we call \emph{split Nakamura manifolds}.

Split Nakamura manifolds are more accessible as they are among those studied by Kasuya in \cite{Kasuya}, where he provides a useful way to compute their Dolbeault cohomology. We will describe their geometry under various aspects. In particular, we will characterize when their Fr\"olicher spectral sequence degenerates at the $E_1$ page (Proposition \ref{prop: frolicher}) and show that this fact is equivalent to the validity of the $\del\delbar$-Lemma (Proposition \ref{prop: del delbar lemma}). Then we concentrate on $p$-K\"ahlerianity of Nakamura manifolds, showing that they are never $p$-K\"ahler unless $p = n$ or $p = n + 1$ (Proposition \ref{prop: p-kahlerianity}). Finally, we address the study of their automorphisms and we show that every automorphism of a split Nakamura manifold can be lifted to an automorphism of $\IC^{n + 1}$ of a very special form, providing necessary and sufficient conditions for a map of such form to be a lift of some automorphism of the split Nakamura manifold in consideration.

Our discussion is far from being complete, as there are many aspects one can address. For example, while we focus mainly on the complex geometric aspects of Nakamura manifolds, in \cite{Lusetti-Tomassini} their symplectic geometry is studied, also in relation with the validity of the Hard Lefschetz Condition.

The structure of the paper is as follows. In section \ref{sect: construction} we give the construction of Nakamura manifolds and give some example. In Section \ref{sect: de Rham} we study the de Rham cohomology, while in Section \ref{sect: dolbeault} the Dolbeault cohomology. In particular, in this Section we compute the Kodaira dimension and we address the degeneration of the Fr\"olicher spectral sequence, the validity of the $\del\delbar$-Lemma and we will dedicate on the study of the cohomology when it is not generated by invariant forms. Section \ref{sect: p-kahlerianity} deals with the $p$-K\"ahler property, while Section \ref{sect: deformations} concerns deformations of Nakamura manifolds and Section \ref{sect: albanese} their Albanese map. The last section, Section \ref{sect: automorphisms} describes the automorphism group of split Nakamura manifolds providing an explicit expression of lifts of automorphisms to the universal cover.

\begin{ack}
The author is a member of the GNSAGA of INdAM and was supported by the project PRIN2022 ``Real and Complex Manifolds: Geometry and Holomorphic Dynamics'' (Project code: 2022AP8HZ9). This research was granted by University of Parma through the action ``Bando di Ateneo 2025 per la ricerca''.

The author warmly acknowledges all the participants of the ``Workshop on Cohomological and metric aspects of Hermitian and almost complex manifolds'' held in Budapest in the period 8$^\text{th}$--12$^\text{th}$ September 2025, where he could discuss and present the results which are the bulk of this paper.
\end{ack}

\section{The construction of Nakamura manifolds}\label{sect: construction}

Fix
\[\lambda_1, \ldots, \lambda_n \in \IR \qquad \text{and} \qquad \tau \in \cH,\]
where $\cH = \set{z \in \IC \st \Im(z) > 0}$ is the complex upper half plane. Define the action of $\IC$ on $\IC^n$
\[\begin{array}{rccl}
\rho = \rho_\tau: & \IC & \longrightarrow & \GL(n, \IC)\\
 & z & \longmapsto & \diag\pa{e^{\lambda_1 D_\tau(w)}, \ldots, e^{\lambda_n D_\tau(w)}},
\end{array}\]
where
\[D_\tau(w) = \frac{\bar{w}\tau - w\bar{\tau}}{\tau - \bar{\tau}}\]
is the first coordinate of $w \in \IC$ with respect to the real basis $\set{1, \tau}$.

We use $\rho_\tau$ to define the semidirect product
\[G = \IC^n \rtimes_{\rho_\tau} \IC,\]
where the group operation is explicitly given by
\[(\beta, \alpha)*(z, w) = \pa{\beta + \rho(\alpha) \cdot z, \alpha + w}, \qquad \beta, z \in \IC^n, \quad \alpha, w \in \IC.\]

\begin{rem}
The identity is $(0, 0)$ and the inverse of the element $(z, w) \in G$ is $(z, w)^{-1} = (-\rho(-w) \cdot z, -w)$.
\end{rem}

\begin{defin}
A \emph{Nakamura manifold} $N$ is any compact complex manifold which is obtained as the quotient of a group $G$ constructed as above by the (left) action of a cocompact lattice $\Gamma \leq G$:
\[N = \Gamma \backslash G.\]
\end{defin}

\begin{rem}
As a matter of notation, we write simply $N$ instead of $N_{\lambda_1, \ldots, \lambda_n; \tau; \Gamma}$ for a Nakamura manifold, omitting the dependence from the choices involved in the construction.
\end{rem}

\subsection{Examples}\label{sect: examples}

We fix the isomorphism of $\IR$-vector spaces
\[\begin{array}{ccl}
\IR^{2n} & \longrightarrow & \IC^n\\
\left( \begin{array}{c} x\\y \end{array} \right) & \longmapsto & z = x + \ii y
\end{array} \qquad \text{for } x, y \in \IR^n,\]
with respect to which the complex structure of $\IC^n$ is given by the matrix
\[J = \left( \begin{array}{c|c}
0 & -\id_n\\
\hline
\id_n & 0
\end{array} \right).\]

Fix a matrix $M \in \SL(2n, \IZ)$ for which there exist a matrix $P \in \GL(2n, \IR)$ and $\lambda_1, \ldots, \lambda_n \in \IR$ such that
\[P \cdot M \cdot P^{-1} = \left( \begin{array}{c|c}
\Delta & 0\\
\hline
0 & \Delta
\end{array} \right), \qquad \text{where } \Delta = \diag\pa{e^{\lambda_1}, \qquad, e^{\lambda_n}}.\]
Write
\[P = \left( \begin{array}{c|c|c}
p_1 & \ldots & p_{2n}\\
\hline
q_1 & \ldots & q_{2n}
\end{array} \right), \qquad \text{with } p_1, \ldots, p_{2n}, q_1, \ldots, q_{2n} \in \IR^n,\]
define
\[B = \left( \begin{array}{c|c} \id_n & \ii \id_n \end{array} \right) \cdot P = \left( \begin{array}{c|c|c} p_1 + \ii q_1 & \ldots & p_{2n} + \ii q_{2n} \end{array} \right)\]
and let $\Gamma' = B \cdot \IZ^{2n} \leq \IC^n$ be the sublattice of $\IC^n$ generated by the columns of $B$.

Let now $\tau \in \cH$ and construct the lattice $\Lambda_\tau = \IZ \oplus \tau \cdot \IZ \leq \IC$. For $\alpha \in \Lambda_\tau$ we have that $D_\tau(\alpha) \in \IZ$ and so
\[\rho_\tau(\alpha) = \diag\pa{e^{\lambda_1 D_\tau(w)}, \ldots, e^{\lambda_n D_\tau(w)}} = \Delta^{D_\tau(\alpha)}.\]
As a consequence, for every $\beta \in \Gamma'$ of the form $\beta = B \cdot c$ for some $c \in \IZ^{2n}$ we have that
\[\rho_\tau(\alpha) \cdot \beta = \left( \begin{array}{c|c} \id_n & \ii \id_n \end{array} \right) \cdot \left( \begin{array}{c|c}
\Delta & 0\\
\hline
0 & \Delta
\end{array} \right)^{D_\tau(\alpha)} \cdot P \cdot c = B \cdot \underbrace{M^{D_\tau(\alpha)} \cdot c}_{\in \IZ^{2n}} \in \Gamma'.\]

As a consequence, the action of $\Lambda_\tau$ on $\IC^n$ via $\rho_\tau$ preserves the lattice $\Gamma'$ and so
\[\Gamma = \Gamma' \rtimes_{\rho_\tau} \Lambda_\tau \leq \IC^n \rtimes_{\rho_\tau} \IC = G\]
is a sublattice, which allows us to construct a Nakamura manifold $N = \Gamma \backslash G$.

\subsubsection{A particular case}\label{sect: from n x n matrix}

Let $M \in \SL(n, \IZ)$ be a matrix for which there exist $P \in \GL(n, \IR)$ and $\lambda_1, \ldots, \lambda_n \in \IR$ such that
\[P \cdot M \cdot P^{-1} = \Delta = \diag\pa{e^{\lambda_1}, \ldots, e^{\lambda_n}}.\]
Let $\tau \in \cH$ and observe that
\[\left( \begin{array}{c|c}
P & \Re(\tau) \cdot P\\
\hline
0 & \Im(\tau) \cdot P
\end{array} \right) \cdot \left( \begin{array}{c|c}
M & 0\\
\hline
0 & M
\end{array} \right) \cdot \left( \begin{array}{c|c}
P & \Re(\tau) \cdot P\\
\hline
0 & \Im(\tau) \cdot P
\end{array} \right)^{-1} = \left( \begin{array}{c|c}
\Delta & 0\\
\hline
0 & \Delta
\end{array} \right).\]
In this case we have that $B = \left( \begin{array}{c|c} P & \tau P \end{array} \right)$ and so $\Gamma' = B \cdot \IZ^{2n} = P \cdot \Lambda_\tau^n$.

Nakamura manifolds of this form were introduced in \cite{Cattaneo-Tomassini} with the choice of $\tau = t \ii$ for some positive $t \in \IR$.

\subsubsection{An isomorphism}

Assume that we have a matrix $M \in \SL(n, \IZ)$ as in Section \ref{sect: from n x n matrix}, for which there exist $\lambda_1, \ldots, \lambda_n \in \IR$ and matrices $P_1, P_2 \in \GL(n, \IR)$ such that
\[P_1 \cdot M \cdot P_1^{-1} = \Delta = P_2 \cdot M \cdot P_2^{-1}, \qquad \Delta = \diag\pa{e^{\lambda_1}, \ldots, e^{\lambda_n}}.\]
Fix $\tau \in \cH$ and consider the lattices $\Lambda_\tau = \IZ \oplus \tau \cdot \IZ$ and
\[\Gamma'_i = P_i \cdot \Lambda_\tau^n, \qquad \Gamma_i = \Gamma'_i \rtimes \Lambda_\tau, \qquad \text{for } i = 1, 2.\]

Define the map
\[\begin{array}{rccc}
F: & \IC^n \rtimes_\rho \IC & \longrightarrow & \IC^n \rtimes_\rho \IC\\
 & (z, w) & \longmapsto & (Rz, w),
\end{array}\]
where $R = P_2 \cdot P_1^{-1}$ and see how $F$ acts on the lattices $\Gamma_1$ and $\Gamma_2$.
\begin{enumerate}
\item First, we observe that
\[\Delta \cdot R = P_2 \cdot M \cdot P_1^{-1} = R \cdot \Delta.\]
\item Let $\beta = P_1 c \in \Gamma'_1$ for some $c \in \Lambda_\tau^n$ and $\alpha \in \Lambda_\tau$, then $(\beta, \alpha) \in \Gamma$ and
\[F(\beta, \alpha) = (R\beta, \alpha) = (P_2 c, \alpha) \in \Gamma_2\]
This means that $F$ is a bijection such that $F(\Gamma_1) = \Gamma_2$.
\item Let $(z, w), (z', w') \in \IC^n \rtimes_\rho \IC$ be such that $(z', w') = (\beta, \alpha)*(z, w)$ for some $(\beta, \alpha) \in \Gamma_1$. Then
\[\begin{array}{rl}
F(z', w') = & F((\beta, \alpha)*(z, w)) =\\
= & F(\beta + \rho(\alpha) \cdot z, \alpha + w) =\\
= & F(\beta + \Delta^{D(\alpha)} \cdot z, \alpha + w) =\\
= & (R\beta + R\Delta^{D(\alpha)} \cdot z, \alpha + w) =\\
= & (R\beta + \Delta^{D(\alpha)}R \cdot z, \alpha + w) =\\
= & (R\beta, \alpha)*(Rz, w) =\\
= & F(\beta, \alpha)*F(z, w).
\end{array}\]
\end{enumerate}
As a consequence $F$ induces a map on the corresponding Nakamura manifolds, say $f: N_1 = \Gamma_1 \backslash C^n \rtimes_\rho \IC \longrightarrow N_2 = \Gamma_2 \backslash C^n \rtimes_\rho \IC$, which is a biholomorphism (of course, the inverse is induced by $F^{-1}$).

This shows that the construction exposed in Section \ref{sect: from n x n matrix} does not depend on the choice of the diagonalizing matrix $P$, but only on $M \in \SL(n, \IZ)$ and $\tau \in \cH$.

\subsection{Faithfulness and freedom of the action \texorpdfstring{$\rho_\tau$}{rho}}\label{sect: faithfulness and freedom}

In this Section we characterize when $\rho_\tau$ is faithful and study the fixed locus of $\rho_\tau(\alpha)$ for $\alpha \in \Lambda_\tau$.

\begin{prop}[Faithfulness]\label{prop: faithfulness}
The action $\rho_\tau$ is never faithful. More precisely
\[\ker(\rho_\tau) = \left\{ \begin{array}{ll}
\IC & \text{if } \lambda_i = 0 \text{ for every } i,\\
\tau \cdot \IR & \text{if } \lambda_i \neq 0 \text{ for some } i.
\end{array} \right.\]
\end{prop}
\begin{proof}
If $\lambda_i = 0$ for every $i \in \set{1, \ldots, n}$, then $\rho$ is trivial and so $\ker(\rho) = \IC$. On the contrary, if $\lambda_i \neq 0$ for some $i \in \set{1, \ldots, n}$ then we see that
\[\ker(\rho) = \set{w \in \IC \st D_\tau(w) = 0} = \tau \cdot \IR \leq \IC.\]
\end{proof}

\begin{rem}
It follows from Proposition \ref{prop: faithfulness} that complex tori are Nakamura manifolds corresponding to $\lambda_i = 0$ for every $i$.
\end{rem}

\begin{prop}[Freedom]\label{prop: freedom}
Assume that $\rho_\tau$ is non-trivial and that $w \in \IC \smallsetminus \ker(\rho_\tau)$. Then
\[\Fix(\rho_\tau(w)) = \Span \set{e_i \st i \text{ satisfies } \lambda_i = 0}.\]
\end{prop}
\begin{proof}
A vector $z = (z_1, \ldots, z_n) \in \IC^n$ is fixed by $\rho(w)$ if and only if $e^{\lambda_i D_\tau(w)}z_i = z_i$ for every $i$, i.e., if and only if $\pa{e^{\lambda_i D_\tau(w)} - 1} z_i = 0$. For any fixed $i \in \set{1, \ldots, n}$ there are two possibilities:
\begin{enumerate}
\item $\lambda_i = 0$. In this case the $i^\text{th}$ equation is trivial and provides no information.
\item $\lambda_i \neq 0$. In this case we deduce that $z_i = 0$ as $D_\tau(w) \neq 0$.
\end{enumerate}
Hence
\[\begin{array}{rl}
\Fix(\rho(w)) = & \set{(z_1, \ldots, z_n) \st z_i = 0 \text{ for every } i \text{ such that } \lambda_i \neq 0} =\\
= & \Span \set{e_i \st i \text{ satisfies } \lambda_i = 0}.
\end{array}\]
\end{proof}

\begin{rem}
Under the assumptions of Proposition \ref{prop: freedom}, we see that $\Fix(\rho_\tau(w))$ does not depend on $w \in \IC \smallsetminus \ker(\rho_\tau)$ and that
\[\dim(\Fix(\rho_\tau(w))) = \#\set{i \in \set{1, \ldots, n} \st \lambda_i = 0} < n.\]
\end{rem}

\begin{cor}
If $\lambda_i \neq 0$ for every $i = \set{1, \ldots, n}$ then
\[\Fix(\rho(w)) = \left\{ \begin{array}{ll}
\IC^n & \text{if } w \in \tau \cdot \IR,\\
\set{0} & \text{if } w \notin \tau \cdot \IR.
\end{array} \right.\]
\end{cor}

\section{de Rham cohomology}\label{sect: de Rham}

In this section we describe the de Rham cohomology of a Nakamura manifold associated to $\lambda_1, \ldots, \lambda_n \in \IR$ and $\tau \in \cH$.

Introduce real coordinates by means of
\[z_i = x_i + \ii y_i, \qquad w = r + \ii s \qquad (i = 1, \ldots, n)\]
and define the forms (on $G$):
\[\begin{array}{lll}
e^0 = dr, & f^0 = ds, & \\
e^i = e^{-D_\tau(w) \lambda_i}ndx_i, & f^i = e^{-D_\tau(w) \lambda_i} dy_i, & i = 1, \ldots, n.
\end{array}\]
These forms are invariant for the action of $\Gamma$ on $G$ by multiplication on the left, hence they descend to forms on $N$ (which will maintain the same name). We can then compute that
\[\begin{array}{lll}
d e^0 = 0, & d e^i = -\lambda_i \pa{e^0 - \frac{\Re(\tau)}{\Im(\tau)} f^0} \wedge e^i, & \\
d f^0 = 0, & d f^i = -\lambda_i \pa{e^0 - \frac{\Re(\tau)}{\Im(\tau)} f^0} \wedge f^i, & i = 1, \ldots, n.
\end{array}\]
Let now $\set{e_0, f_0, e_1, f_1, \ldots, e_n, f_n}$ be the basis of vector fields dual to the basis $\set{e^0, f^0, e^1, f^1, \ldots, e^n, f^n}$, it follows that the only non-trivial commutators among these elements are
\[\begin{array}{lll}
[e_0, e_i] = \lambda_i e_i, & [f_0, e_i] = -\frac{\Re(\tau)}{\Im(\tau)} \lambda_i e_i, & \\
{[e_0, f_i]} = \lambda_i f_i, & [f_0, f_i] = -\frac{\Re(\tau)}{\Im(\tau)} \lambda_i f_i, & i = 1, \ldots, n.
\end{array}\]
As a consequence, if we call $\gothg$ the Lie algebra of $G$ (which is generated by the vectors $e_0, f_0, e_1, f_1, \ldots, e_n, f_n$) we see that
\[{[\gothg, \gothg]} = \Span \set{e_i, f_i \st i \text{ satisfies } \lambda_i \neq 0} \leq \Span\set{e_1, f_1, \ldots, e_n, f_n}\]
and so
\[{[[\gothg, \gothg], \gothg]} = [\gothg, \gothg] \qquad \text{and} \qquad [[\gothg, \gothg], [\gothg, \gothg]] = 0.\]
This shows the following lemma.

\begin{lemma}
The Lie algebra $\gothg$ is solvable, and it is not nilpotent unless $\lambda_i = 0$ for every $i = 1, \ldots, n$.
\end{lemma}

Consider now a vector
\[v = \sum_{i = 0}^n \mu_i e_i + \nu_i f_i \in \gothg\]
and the corresponding adjoint morphism $\operatorname{ad}(v) = [v, \blank]$. The matrix representing it with respect to our basis is
\[
\resizebox{\textwidth}{!}{$
\setlength{\arraycolsep}{3pt}
\left(
\begin{array}{cc|ccc|ccc}
0 & 0 & 0 & \ldots & 0 & 0 & \ldots & 0\\
0 & 0 & 0 & \ldots & 0 & 0 & \ldots & 0\\
\hline
-\lambda_1 \mu_1 & \frac{\Re(\tau)}{\Im(\tau)}\lambda_1 \mu_1 & 
\lambda_1 \pa{\mu_0 - \frac{\Re(\tau)}{\Im(\tau)} \mu_0} & \ldots & 0 & 0 & \ldots & 0\\
\vdots & \vdots & \vdots & \ddots & \vdots & \vdots & \ddots & \vdots\\
-\lambda_n \mu_n & \frac{\Re(\tau)}{\Im(\tau)}\lambda_n \mu_n & 
0 & \ldots & \lambda_n \pa{\mu_0 - \frac{\Re(\tau)}{\Im(\tau)} \mu_0} & 0 & \ldots & 0\\
\hline
-\lambda_1 \nu_1 & \frac{\Re(\tau)}{\Im(\tau)}\lambda_1 \nu_1 & 
0 & \ldots & 0 & \lambda_1 \pa{\mu_0 - \frac{\Re(\tau)}{\Im(\tau)} \mu_0} & \ldots & 0\\
\vdots & \vdots & \vdots & \ddots & \vdots & \vdots & \ddots & \vdots\\
-\lambda_n \nu_n & \frac{\Re(\tau)}{\Im(\tau)}\lambda_n \nu_n & 
0 & \ldots & 0 & 0 & \ldots & \lambda_n \pa{\mu_0 - \frac{\Re(\tau)}{\Im(\tau)} \mu_0}
\end{array}
\right)
$}
\]
which is lower triangular. As a consequence we easily see that all the eigenvalues are real, i.e., that $\gothg$ is completely solvable.

By Hattori's Theorem \cite[Corollary 4.2]{Hattori} the de Rham cohomology of a Nakamura manifold $N$ is the Chevalley--Eilenberg cohomology of the Lie algebra $\gothg$.

\begin{prop}\label{prop: sum is zero}
Let $N = \Gamma \backslash G$ be a Nakamura manifold, corresponding to $\lambda_1, \ldots, \lambda_n \in \IR$. Then
\[\sum_{i = 1}^n \lambda_i = 0.\]
\end{prop}
\begin{proof}
Since $G$ contains a cocompact lattice, it is unimodular. In particular, we have that $\tr(\operatorname{ad}(e_0)) = 0$. Using the representative matrix above we see that $\tr(\operatorname{ad}(e_0)) = 2 \sum_{i = 1}^n \lambda_i$, from which the result follows.
\end{proof}

A direct consequence of this fact is the following: we have that
\[\det(\Delta) = \det\pa{\diag\pa{e^{\lambda_1}, \ldots, e^{\lambda_n}}} = e^{\sum_{i = 1}^n \lambda_i} = 1,\]
which justifies that in Section \ref{sect: examples} we chose our matrix $M$ in $\SL(n, \IZ)$.

\section{Dolbeault cohomology}\label{sect: dolbeault}
\subsection{Invariant \texorpdfstring{$(1, 0)$}{(1, 0)}-forms}
Let $N = \Gamma \backslash G$ be a Nakamura manifold corresponding to a choice of $\lambda_1, \ldots, \lambda_n \in \IR$ and $\tau \in \cH$.

Given $(\beta, \alpha) \in \Gamma$, the left translation by $(\beta, \alpha)$ is the map
\[\begin{array}{rccl}
L_{(\beta, \alpha)}: & G & \longrightarrow & G\\
 & (z, w) & \longmapsto & (\beta, \alpha)*(z, w) = (\beta + \rho_\tau(\alpha) \cdot z, \alpha + w).
\end{array}\]

Consider the following $(1, 0)$-forms on $G$:
\[\varphi^0 = dw, \qquad \varphi^i = e^{-D_\tau(w)\lambda_i}dz_i \quad (\text{for } i = 1, \ldots, n),\]
it is then immediate to verify that they are invariant under the action of left translations by elements of $\Gamma$, hence that they descend to well defined $(1, 0)$-forms on $N$, with the same names.

Moreover, we can compute the structure equations:
\[d\varphi^0 = 0, \qquad d\varphi^i = \frac{\lambda_i}{\tau - \bar{\tau}}\pa{\bar{\tau}\varphi^0 - \tau \bar{\varphi}^0} \wedge \varphi^i \quad (\text{for } i = 1, \ldots, n).\]
In particular, we see that
\[\begin{array}{lll}
\delbar \varphi^0 = 0, & \delbar \varphi^i = \lambda_i \frac{\tau}{\tau - \bar{\tau}} \varphi^i \wedge \bar{\varphi}^0, & \\
\delbar \bar{\varphi}^0 = 0 & \delbar \bar{\varphi}^i = -\lambda_i \frac{\tau}{\tau - \bar{\tau}} \bar{\varphi}^0 \wedge \bar{\varphi}^i, & (\text{for } i = 1, \ldots, n).
\end{array}\]

\subsection{Holomorphic \texorpdfstring{$(n + 1, 0)$}{(n + 1, 0)}-forms and Kodaira dimension}

It follows from the previous section that Nakamura manifolds are never holomorphically parallelizable (i.e., with holomorphically trivial tangent bundle), unless they are complex tori.

\begin{prop}\label{prop: trivial canonical}
The canonical bundle $\omega_N$ of a Nakamura manifold $N$ is holomorphically trivial.
\end{prop}
\begin{proof}
The $(n + 1, 0)$-form
\[\psi = \varphi^0 \wedge \varphi^1 \wedge \ldots \wedge \varphi^n\]
is a smooth trivialization for $\omega_N$. To show that it is holomorphic, we compute that
\[d\psi = \delbar \psi = \frac{\tau}{\tau - \bar{\tau}} \pa{\sum_{i = 1}^n \lambda_i} \varphi^0 \wedge \bar{\varphi}^0 \wedge \varphi^1 \wedge \ldots \wedge \varphi^n = 0\]
by Proposition \ref{prop: sum is zero}.
\end{proof}

\begin{cor}
A Nakamura manifold $N$ has $\kod(N) = 0$.
\end{cor}

\subsection{Dolbeault cohomology and \texorpdfstring{$\del\delbar$}{del-delbar}-Lemma}

From this section we will assume that $N = \Gamma \backslash G$ is a Nakamura manifold corresponding to a choice of $\Gamma$ which is compatible with the structure of $G$ as semidirect product.

\begin{defin}
A Nakamura manifold $N = \Gamma \backslash G$ corresponding to $\lambda_1, \ldots, \lambda_n \in \IR$ and $\tau \in \cH$ is called a \emph{split Nakamura manifold} if
\[\Gamma = \Gamma' \rtimes_{\rho_\tau} \Lambda_\tau.\]
\end{defin}

\begin{exam}
All the examples described in Section \ref{sect: examples} are split Nakamura manifolds.
\end{exam}

This requirement allows us to use \cite[Corollary 4.2]{Kasuya} to compute the Dolbeault cohomology of $N$ as follows.

The Dolbeault cohomology of a split Nakamura manifold $N$ is computed by the subcomplex $B^{\bullet, \bullet}$ of $\Lambda^{\bullet, \bullet} N$ such that $B^{p, q}$ is the $\IC$-vector space generated by the forms of type
\[\begin{array}{llll}
f_{IJ}(w) \cdot \varphi^I \wedge \bar{\varphi}^J & \text{with} & |I| = p, & |J| = q,\\
f_{IJ}(w) \cdot \varphi^0 \wedge \varphi^I \wedge \bar{\varphi}^J & \text{with} & |I| = p - 1, & |J| = q,\\
f_{IJ}(w) \cdot \bar{\varphi}^0 \wedge \varphi^I \wedge \bar{\varphi}^J & \text{with} & |I| = p, & |J| = q - 1,\\
f_{IJ}(w) \cdot \varphi^0 \wedge \bar{\varphi}^0 \wedge \varphi^I \wedge \bar{\varphi}^J & \text{with} & |I| = p - 1, & |J| = q - 1,\\
\end{array}\]
where $I, J$ are the multi-indices for which the function
\[f_{IJ}(w) = e^{-\ii c_{IJ} \frac{\Re(\bar{\tau}w)}{\Im(\tau)}}, \qquad c_{IJ} = \sum_{\sigma = 1}^{|I|} \lambda_{i_\sigma} + \sum_{\xi = 1}^{|J|} \lambda_{i_\xi}\]
satisfies $f_{IJ}(\alpha) = 1$ for every $\alpha \in \Lambda_\tau$.

\begin{prop}\label{prop: iso cohomology}
For every $p \geq 0$ the complex $\pa{B^{p, \bullet}, \delbar}$ has trivial differential. As a consequence
\[H^{p, q}_{\delbar}(N) \simeq B^{p, q}.\]
\end{prop}
\begin{proof}
It is easy to see that
\[\delbar f_{IJ} = \frac{\tau}{\tau - \bar{\tau}} c_{IJ} f_{IJ} \bar{\varphi}^0, \qquad \delbar \pa{\varphi^I \wedge \bar{\varphi}^J} = -\frac{\tau}{\tau - \bar{\tau}} c_{IJ} \varphi^I \wedge \bar{\varphi}^J.\]
As a consequence
\[\delbar\pa{f_{IJ} \cdot \varphi^I \wedge \bar{\varphi}^J} = 0,\]
hence
\[\begin{array}{l}
\delbar\pa{f_{IJ}(w) \cdot \varphi^0 \wedge \varphi^I \wedge \bar{\varphi}^J} = 0,\\
\delbar\pa{f_{IJ}(w) \cdot \bar{\varphi}^0 \wedge \varphi^I \wedge \bar{\varphi}^J} = 0,\\
\delbar\pa{f_{IJ}(w) \cdot \varphi^0 \wedge \bar{\varphi}^0 \wedge \varphi^I \wedge \bar{\varphi}^J} = 0.
\end{array}\]
\end{proof}

\begin{rem}\label{rem: del}
A similar computation shows that
\[\begin{array}{l}
\del\pa{f_{IJ}(w) \cdot \varphi^I \wedge \bar{\varphi}^J} = 2 c_{IJ} \frac{\bar{\tau}}{\tau - \bar{\tau}} f_{IJ} \varphi^0 \wedge \varphi^I \wedge \bar{\varphi}^J,\\
\del\pa{f_{IJ}(w) \cdot \varphi^0 \wedge \varphi^I \wedge \bar{\varphi}^J} = 0,\\
\del\pa{f_{IJ}(w) \cdot \bar{\varphi}^0 \wedge \varphi^I \wedge \bar{\varphi}^J} = 2 c_{IJ} \frac{\bar{\tau}}{\tau - \bar{\tau}} f_{IJ} \varphi^0 \wedge \bar{\varphi}^0 \wedge \varphi^I \wedge \bar{\varphi}^J,\\
\del\pa{f_{IJ}(w) \cdot \varphi^0 \wedge \bar{\varphi}^0 \wedge \varphi^I \wedge \bar{\varphi}^J} = 0.
\end{array}\]
\end{rem}

\begin{prop}\label{prop: frolicher}
The Fr\"olicher spectral sequence of a split Nakamura manifold $N$ degenerates at the $E_1$ page if and only if for each pair of multi-indices $I, J$ for which $f_{IJ}(\alpha) = 1$ for every $\alpha \in \Lambda_\tau$ it holds that $c_{IJ} = 0$.
\end{prop}
\begin{proof}
The $E_1$ page of the Fr\"olicher spectral sequence is given by $\del: H^{p, q}_{\delbar}(N) \longrightarrow H^{p + 1, q}_{\delbar}(N)$. The result then follows from the isomorphism $H^{p, q}_{\delbar}(N) \simeq B^{p, q}$ of Proposition \ref{prop: iso cohomology} and Remark \ref{rem: del}.
\end{proof}

\begin{cor}\label{cor: dolbeault}
Assume that $N$ is a split Nakamura manifold whose Fr\"olicher spectral sequence degenerates at the $E_1$ page. Then
\begin{enumerate}
\item all the invariant $(p, q)$-forms corresponding to multi-indices $I, J$ such that $c_{IJ} = 0$ are $d$-closed;
\item the Dolbeault cohomology of $N$ can be computed by means of invariant forms, namely
\[H^{p,q}_{\delbar}(N)=\Span\left\{
\begin{aligned}
\varphi^I \wedge \bar{\varphi}^J
&\quad \text{with } |I|=p,\ |J|=q,\ c_{IJ}=0,\\
\varphi^0 \wedge \varphi^I \wedge \bar{\varphi}^J
&\quad \text{with } |I|=p-1,\ |J|=q,\ c_{IJ}=0,\\
\bar{\varphi}^0 \wedge \varphi^I \wedge \bar{\varphi}^J
&\quad \text{with } |I|=p,\ |J|=q-1,\ c_{IJ}=0,\\
\varphi^0 \wedge \bar{\varphi}^0 \wedge \varphi^I \wedge \bar{\varphi}^J
&\quad \text{with } |I|=p-1,\ |J|=q-1,\ c_{IJ}=0
\end{aligned}
\right\}.\]
\end{enumerate}
\end{cor}

Finally, we observe that, by Proposition \ref{prop: frolicher}, the condition that the Fr\"olicher spectral sequence of $N$ degenerates at the $E_1$ page is equivalent to \cite[Condition 3.3]{Kasuya1}. This leads us to the following result, which is an application of \cite[Theorem 3.4]{Kasuya1}.

\begin{prop}\label{prop: harmonic}
Assume that the Fr\"olicher spectral sequence of a split Nakamura manifold $N$ degenerates at the $E_1$ page. Then $N$ satisfies the $\del\delbar$-Lemma and the Dolbeault cohomology spaces can be viewed as the spaces of harmonic forms for the metric
\[g = dw \tensor d\bar{w} + \sum_{i = 1}^ n e^{-2D_\tau(w)\lambda_i} dz_i \tensor d\bar{z}_i.\]
\end{prop}

\begin{rem}\label{rem: balanced metric}
Consider the vector fields
\[V_0 = \partial_w, \qquad V_i = e^{D_\tau(w)\lambda_i} \partial_{z_i}, \quad (i = 1, \ldots, n),\]
which are the dual fields of $\varphi^0, \varphi^1, \ldots, \varphi^n$. The metric in Proposition \ref{prop: harmonic} is the one for which these vector fields define an orthonormal system. It is easy to see that the fundamental $(1, 1)$-form of this metric is
\[\omega_g = \sum_{i = 0}^n \varphi^i \wedge \bar{\varphi}^i,\]
hence that $\omega_g$ is an invariant balanced metric (i.e., $d \omega_g^n = 0$).
\end{rem}

\begin{prop}[$\del\delbar$-Lemma]\label{prop: del delbar lemma}
A split Nakamura manifold $N$ satisfies the $\del\delbar$-Lemma if and only if its Fr\"olicher spectral sequence degenerates at the $E_1$ page.
\end{prop}
\begin{proof}
By \cite[Paragraph (5.21)]{DGMS}, the degeneration of the Fr\"olicher spectral sequence at the $E_1$ page is a necessary condition for the $\del\delbar$-Lemma to hold. It is also sufficient by Proposition \ref{prop: harmonic}.
\end{proof}

\begin{rem}
Using Proposition \ref{prop: frolicher} it is possible to deduce that the degeneration of the Fr\"olicher spectral sequence at $E_1$ implies the $\del\delbar$-Lemma also using \cite[Corollary 8 and Definition 4]{ASTT}: the two requirements (H1) and (H2) in \cite[Definition 4]{ASTT} are a direct consequence of Corollary \ref{cor: dolbeault}.
\end{rem}

Passing on the $\IZ$-basis $\set{1, \tau}$ for $\Lambda_\tau$ we see that the condition that $f_{IJ}(\alpha) = 1$ for every $\alpha \in \Lambda_\tau$ is equivalent to
\[c_{IJ} \frac{\Re(\tau)}{\Im(\tau)} \in 2\pi\IZ \qquad \text{and} \qquad c_{IJ} \frac{|\tau|^2}{\Im(\tau)} \in 2\pi\IZ.\]

\begin{prop}\label{prop: non rational}
Let $\tau \in \cH$ be such that $\frac{\Re(\tau)}{|\tau|^2} \notin \IQ$. Then the only multi-indices $I, J$ which provides generators for $B^{p, q}$ are those for which $c_{IJ} = 0$. In particular, in this case the Dolbeault cohomology of the split Nakamura manifold $N$ is generated by invariant forms and the $\del\delbar$-Lemma holds.
\end{prop}
\begin{proof}
Assume that $I, J$ are multi-indices for which $c_{IJ} \neq 0$ and assume that
\[c_{IJ} \frac{\Re(\tau)}{\Im(\tau)} = 2h\pi \in 2\pi\IZ \qquad \text{and} \qquad c_{IJ} \frac{|\tau|^2}{\Im(\tau)} = 2k\pi \in 2\pi\IZ\]
for suitable $h, k \in \IZ$. Then $k \neq 0$ and
\[c_{IJ} \frac{\Re(\tau)}{\Im(\tau)} \cdot k = h \cdot 2\pi \cdot k = h \cdot c_{IJ} \frac{|\tau|^2}{\Im(\tau)}.\]
As a consequence
\[\frac{\Re(\tau)}{|\tau|^2} = \frac{h}{k} \in \IQ.\]
\end{proof}

\begin{rem}\label{rem: special tau}
All the complex numbers $\tau \in \cH$ such that $\frac{\Re(\tau)}{|\tau|^2} \in \IQ$ are of the form 
\[\tau = ql^2 + \ii l \sqrt{1 - q^2 l^2}\]
for some $l \in \IR_{> 0}$ and $q \in \IQ$ such that $l|q| < 1$. With this form we have $|\tau| = l$ and $\frac{\Re(\tau)}{|\tau|^2} = q$.
\end{rem}

\subsection{Cohomology with non-invariant forms}

Corollary \ref{cor: dolbeault} allows us to compute explicitly the Dolbeault cohomology and the Hodge numbers of a split Nakamura manifold satisfying the $\del\delbar$-Lemma. In this Section we want to address the complementary situation.

Assume that $N$ is a split Nakamura manifold which does not satisfy the $\del\delbar$-Lemma. By Proposition \ref{prop: non rational} we have that $\frac{\Re(\tau)}{|\tau|^2} \in \IQ$ and, moreover, there must exist multi-indices $I, J$ such that $c_{IJ} \neq 0$ and
\[c_{IJ} \frac{\Re(\tau)}{\Im(\tau)} = 2h\pi \in 2\pi\IZ \qquad \text{and} \qquad c_{IJ} \frac{|\tau|^2}{\Im(\tau)} = 2k\pi \in 2\pi\IZ\]
for suitable $h, k \in \IZ$ such that $k \neq 0$ and $\sgn(c_{IJ}) = \sgn(k)$.

The following proposition shows that these conditions actually determine $\tau$ uniquely.

\begin{lemma}\label{lemma: form of tau}
Let $\tau \in \cH$ be such that
\[\frac{\Re(\tau)}{|\tau|^2} \in \IQ, \qquad c \frac{\Re(\tau)}{\Im(\tau)} = 2h\pi, \qquad c \frac{|\tau|^2}{\Im(\tau)} = 2k\pi\]
for suitable $c \in \IR$ and $h,k \in \IZ$ such that $c \neq 0$, $k \neq 0$ and $kc > 0$. Then
\[\tau = \frac{2k\pi}{4\pi^2 h^2 + c^2} (2h\pi + \ii c).\]
\end{lemma}
\begin{proof}
By Remark \ref{rem: special tau} we have that $\tau$ is of the form
\[\tau = ql^2 + \ii l \sqrt{1 - q^2 l^2}\]
for some $q \in \IQ$ and $l \in \IR_{> 0}$ such that $|q|l < 1$. We have then to solve the system in $q, l$
\[\left\{ \begin{array}{l}
c\frac{ql}{\sqrt{1 - q^2 l^2}} = 2h\pi\\
c\frac{l}{\sqrt{1 - q^2 l^2}} = 2k\pi.
\end{array} \right.\]
It follows that $\frac{\Re(\tau)}{|\tau|^2} = q = \frac{h}{k}$, and substituting this in the second equation we find that $l = \frac{2\pi|k|}{\sqrt{4 \pi^2 h^2 + c^2}}$. We observe that
\[|q|l = \frac{|h|}{|k|} \cdot \frac{2\pi |k|}{\sqrt{4\pi^2 h^2 + c^2}} = \frac{2\pi |h|}{\sqrt{4\pi^2 h^2 + c^2}} < \frac{2\pi |h|}{\sqrt{4\pi^2 h^2}} = 1\]
and the result follows.
\end{proof}

It is possible that different choices of $c, h, k$ as in Lemma \ref{lemma: form of tau} give rise to the same value of $\tau$. We want to characterize these different choices.

Let
\[D = \set{(c, h, k) \in \IR \times \IZ \times \IZ \st c \neq 0, k \neq 0, ck > 0}\]
and define the function
\[\begin{array}{rccl}
\tau: & D & \longrightarrow & \cH\\
 & (c, h, k) & \longmapsto & \frac{2k\pi}{4\pi^2 h^2 + c^2} (2h\pi + \ii c)
\end{array}\]
Fix $\tau_0 = \tau(c_0, h_0, k_0)$: we want to describe the fibre of $\tau$ over $\tau_0$.

\begin{lemma}\label{lemma: preimages of tau}
Let $\tau_0 = \tau(c_0, h_0, k_0)$ for some $(c_0, h_0, k_0) \in D$. For $(c, h, k) \in D$ we have that $\tau(c, h, k) = \tau_0$ if and only if $(c, h, k)$ and $(c_0, h_0, k_0)$ are linearly dependent over $\IQ$.
\end{lemma}
\begin{proof}
Assume that $\tau(c, h, k) = \tau_0$. Then
\[\frac{c}{k} = 2\pi \frac{\Im(\tau_0)}{|\tau_0|^2} = \frac{c_0}{k_0}, \qquad \frac{h}{k} = \frac{\Re(\tau_0)}{|\tau_0|^2} = \frac{h_0}{k_0}.\]
As a consequence $(c, h, k)$ and $(c_0, h_0, k_0)$ are linearly dependent over $\IQ$, in fact $(c, h, k) = \frac{k}{k_0} (c_0, h_0, k_0)$. The other implication is trivial.
\end{proof}

\begin{rem}\label{rem: different c, h, k}
Fix $(c, h, k) \in D$ and let $d = \gcd(h, k)$. Let $(c_0, h_0, k_0) = \frac{1}{d}(c, h, k)$: then $(c_0, h_0, k_0) \in D$ and by Lemma \ref{lemma: form of tau}
\[\tau(c, h, k) = \tau(c_0, h_0, k_0) = \tau_0.\]
Moreover, by Lemma \ref{lemma: form of tau} again, for every triple $(\tilde{c}, \tilde{h}, \tilde{k}) \in D$ such that $\tau(\tilde{c}, \tilde{h}, \tilde{k}) = \tau_0$ there exists $m \in \IZ$ such that
\[(\tilde{c}, \tilde{h}, \tilde{k}) = m(c_0, h_0, k_0) = \frac{m}{d}(c, h, k).\]
\end{rem}

We can now go back to our geometric setting. Assume that we are dealing with a split Nakamura manifold for which the $\del\delbar$-Lemma does not hold. As we have already noticed, we have $\frac{\Re(\tau)}{|\tau|^2} \in \IQ$ and there exist  multi-indices $I, J$ such that $c_{IJ} \neq 0$ and
\[c_{IJ} \frac{\Re(\tau)}{\Im(\tau)} = 2h \pi, \qquad c_{IJ} \frac{|\tau|^2}{\Im(\tau)} = 2k \pi\]
for suitable $h, k \in \IZ$ such that $k \neq 0$ and $c_{IJ}k > 0$. It follows by Remark \ref{rem: different c, h, k} that all the multi-indices $H, K$ such that
\[c_{HK} \frac{\Re(\tau)}{\Im(\tau)}, c_{HK} \frac{|\tau|^2}{\Im(\tau)} \in 2 \pi \IZ,\]
i.e., the multi-indices which correspond to generators of some $B^{p, q}$, are all those for which there exists $m \in \IZ$ such that
\[c_{HK} = \frac{m}{\gcd(h, k)} c_{IJ}.\]

\section{\texorpdfstring{$p$}{p}-K\"ahlerianity of Nakamura manifolds}\label{sect: p-kahlerianity}

Recall that a complex manifold $X$ is $p$-K\"ahler (for some $1 \leq p \leq \dim(X)$) if it admits a closed transverse $(p, p)$-form, namely, a closed real $(p, p)$-form $\Omega$ such that
\[\ii^{p^2} 2^{-p} \Omega \wedge \psi \wedge \bar{\psi}\]
is a strictly positive multiple of a volume form for every non-zero simple $(n - p, 0)$-form $\psi$ (see, e.g., \cite[$\S$5]{Alessandrini-Bassanelli}).

\begin{rem}
Nakamura manifolds admit a balanced metric by Proposition \ref{prop: harmonic} and Remark \ref{rem: balanced metric} (), hence all Nakamura manifolds $N = \Gamma \backslash G$ with $G = \IC^n \rtimes_{\rho_\tau} \IC$ are $n$-K\"ahler. On the other hand, as no Nakamura manifold is K\"ahler unless it is a torus, no Nakamura manifolds except tori are $1$-K\"ahler.
\end{rem}

\begin{prop}[\texorpdfstring{$p$}{p}-K\"ahlerianity]\label{prop: p-kahlerianity}
Let $N$ be a Nakamura manifold, corresponding to parameters $\lambda_1, \ldots, \lambda_n \in \IR$ and $\tau \in \cH$. Assume that $N$ is not a torus, i.e., that $\lambda_i \neq 0$ for some $i$. Then $N$ is not $p$-K\"ahler for any $p \in \set{1, \ldots, n - 1}$.
\end{prop}
\begin{proof}
Fix $p \in \set{1, \ldots, n}$. As $N$ is not a torus, by \cite[Lemma 4.3]{Cattaneo-Tomassini} there exists a multi-index $I$ of length $n - p$ such that $\sum_{i_j \in I} i_j \neq 0$. Observe that
\[\begin{array}{c}
d\pa{\varphi^0 \wedge \varphi^{i_1} \wedge \bar{\varphi}^{i_1} \wedge \ldots \wedge \varphi^{i_{n - p}} \wedge \bar{\varphi}^{i_{n - p}}} =\\
= \frac{\tau}{\tau - \bar{\tau}} c_{II} \varphi^0 \wedge \bar{\varphi}^0 \wedge \varphi^{i_1} \wedge \bar{\varphi}^{i_1} \wedge \ldots \wedge \varphi^{i_{n - p}} \wedge \bar{\varphi}^{i_{n - p}}.
\end{array}\]
Let $\theta = \varphi^0 \wedge \varphi^I$: it follows that $\ii^{p^2} 2^{-p} \theta \wedge \bar{\theta}$ is a non-zero simple $d$-exact $(n - p, n - p)$-form on $N$, hence that $N$ can not be $p$-K\"ahler.

Indeed, if $\Omega$ were a closed transverse $(p, p)$-form, then (compare with \cite[Proposition 4.4]{Cattaneo-Tomassini}) the quantity $\int_M \ii^{p^2} 2^{-p} \Omega \wedge \theta \wedge \bar{\theta}$ on one hand is strictly positive by the transversality of $\Omega$, while on the other hand it vanishes by Stokes' Theorem.
\end{proof}

\section{Deformations of Nakamura manifolds}\label{sect: deformations}

Let $N$ be a split Nakamura manifold. Since its canonical bundle is trivial (cf.\ Proposition \ref{prop: trivial canonical}) it follows from Serre duality that
\[H^1(N, TN) \simeq H^n\pa{N, \Omega^1_N}^* \simeq \pa{H^{1, n}_{\delbar}(N)}^*.\]

By Proposition \ref{prop: iso cohomology} we can easily compute the Hodge number $h^{1, n}_{\delbar}(N)$:
\begin{enumerate}
\item the form $f \cdot \varphi^0 \wedge \bar{\varphi}^1 \wedge \ldots \wedge \bar{\varphi}^n$ always provides a generator for $H^{1, n}_{\delbar}(N)$ as in this case $f(w) \equiv 1$ by Proposition \ref{prop: sum is zero}.
\item the form $f_j \cdot \varphi^0 \wedge \bar{\varphi}^0 \wedge \bar{\varphi}^1 \wedge \ldots \wedge \widehat{\bar{\varphi}^j} \wedge \ldots \wedge \bar{\varphi}^n$ has $f_j(w) = e^{\ii \lambda_j \frac{\Re(\bar{\tau} w)}{\Im(\tau)}}$ and so it provides a generator for $H^{1, n}_{\delbar}(N)$ if and only if $-\lambda_j \frac{\Re(\tau)}{\Im(\tau)}, -\lambda_j \frac{|\tau|^2}{\Im(\tau)} \in 2\pi \IZ$.
\item the form $f_i \cdot \varphi^i \wedge \bar{\varphi}^1 \ldots \wedge \bar{\varphi}^n$ has $f_i(w) = e^{-\ii \lambda_i \frac{\Re(\bar{\tau} w)}{\Im(\tau)}}$ and so it provides a generator for $H^{1, n}_{\delbar}(N)$ if and only if $\lambda_i \frac{\Re(\tau)}{\Im(\tau)}, \lambda_i \frac{|\tau|^2}{\Im(\tau)} \in 2\pi \IZ$.
\item the form $f_{ij} \cdot \varphi^0 \wedge \bar{\varphi}^0 \wedge \bar{\varphi}^1 \wedge \ldots \wedge \widehat{\bar{\varphi}^i} \wedge \ldots \wedge \bar{\varphi}^n$ has $f_{ij}(w) = e^{-\ii (\lambda_i - \lambda_j) \frac{\Re(\bar{\tau} w)}{\Im(\tau)}}$ and so it provides a generator for $H^{1, n}_{\delbar}(N)$ if and only if $(\lambda_i - \lambda_j) \frac{\Re(\tau)}{\Im(\tau)}, (\lambda_i - \lambda_j) \frac{|\tau|^2}{\Im(\tau)} \in 2\pi \IZ$.
\end{enumerate}
As a consequence
\begin{equation}\label{eq: h^(1, n)}
\begin{array}{rl}
h^{1. n}_{\delbar}(N) = & 1 + 2 \cdot \#\set{i \st \lambda_i \frac{\Re(\tau)}{\Im(\tau)}, \lambda_i \frac{|\tau|^2}{\Im(\tau)} \in 2\pi \IZ} +\\
 & + \#\set{(i, j) \st (\lambda_i - \lambda_j) \frac{\Re(\tau)}{\Im(\tau)}, (\lambda_i - \lambda_j) \frac{|\tau|^2}{\Im(\tau)} \in 2\pi \IZ} =\\
= & 1 + 2 \cdot \#\set{i \st \lambda_i = 0} + \#\set{(i, j) \st \lambda_i = \lambda_j} +\\
 & + 2 \cdot \#\set{i \st \lambda_i \neq 0 \text{ and } \lambda_i \frac{\Re(\tau)}{\Im(\tau)}, \lambda_i \frac{|\tau|^2}{\Im(\tau)} \in 2\pi \IZ} +\\
 & + \#\set{(i, j) \st \lambda_i \neq \lambda_j \text{ and } (\lambda_i - \lambda_j) \frac{\Re(\tau)}{\Im(\tau)}, (\lambda_i - \lambda_j) \frac{|\tau|^2}{\Im(\tau)} \in 2\pi \IZ} =\\
= & 1 + n + 2 \cdot \#\set{i \st \lambda_i = 0} + 2 \cdot \#\set{(i, j) \st i < j \text{ and } \lambda_i = \lambda_j} +\\
 & + 2 \cdot \#\set{i \st \lambda_i \neq 0 \text{ and } \lambda_i \frac{\Re(\tau)}{\Im(\tau)}, \lambda_i \frac{|\tau|^2}{\Im(\tau)} \in 2\pi \IZ} +\\
 & + \#\set{(i, j) \st \lambda_i \neq \lambda_j \text{ and } (\lambda_i - \lambda_j) \frac{\Re(\tau)}{\Im(\tau)}, (\lambda_i - \lambda_j) \frac{|\tau|^2}{\Im(\tau)} \in 2\pi \IZ}.
\end{array}
\end{equation}

\begin{prop}
Let $N$ be a split Nakamura manifold satisfying the $\del\delbar$-Lemma. Then $N$ has unobstructed deformations and
\[\dim\pa{\Def(N)} = 1 + n + 2 \cdot \#\set{i \st \lambda_i = 0} + 2 \cdot \#\set{(i, j) \st i < j \text{ and } \lambda_i = \lambda_j}.\]
\end{prop}
\begin{proof}
Since $N$ satisfies the $\del\delbar$-Lemma and $\omega_N \simeq \cO_N$ we have that it has unobstructed deformations and $\dim\pa{\Def(N)} = h^{1, n}_{\delbar}(N)$. The computation for this Hodge number made in \eqref{eq: h^(1, n)} simplifies to the one in the statement by Corollary \ref{cor: dolbeault}.
\end{proof}

\section{The Albanese morphism}\label{sect: albanese}

In this Section we provide a general Lemma on the Albanese map, then we focus on the case of Nakamura manifolds.

\begin{prop}\label{prop: albanese general}
Let $X$ be a compact connected complex manifold and $B$ a complex torus. Assume that there is a holomorphic map $\pi: X \longrightarrow B$ such that every $b \in B$ has an open neighborhood $U$ on which $\pi|_{\pi^{-1(U)}}: \pi^{-1}(U) \longrightarrow U$ is isomorphic to $U \times E \longrightarrow U$ for a suitable connected complex manifold $E$. If $\dim(B) \geq \dim\pa{H^0\pa{\pa{X, \Omega^1_X}}}$, then $\pi$ is the Albanese map of $X$ and $\dim(B) = \dim\pa{H^0\pa{\pa{X, \Omega^1_X}}}$.
\end{prop}
\begin{proof}
By \cite[Proposition 1.2.2]{Blanchard} there is an exact sequence
\[\Alb(E) \xrightarrow{i_*} \Alb(X) \xrightarrow{\pi_*} B \longrightarrow 0\]
of complex Lie groups. From this sequence and \cite[Lemma 9.22]{Ueno} we deduce that
\[\dim(B) \geq \dim\pa{H^0\pa{X, \Omega^1_X}} \geq \dim\pa{\Alb(X)} \geq \dim(B)\]
and so all these inequalities are equalities.

Let $K = \ker(\pi_*)$. As $\dim(\Alb(X)) = \dim(B)$ we have that $K$ is finite, and we have the two exact sequences
\begin{enumerate}
\item $\Alb(E) \xrightarrow{i_*} K \longrightarrow 0$, from which $K = i_*(\Alb(E)) = \set{0}$ as $\Alb(E)$ is connected;
\item $0 \longrightarrow K \longrightarrow \Alb(X) \xrightarrow{\pi_*} B \longrightarrow 0$ from which we deduce that $\pi_*$ is an isomorphism.
\end{enumerate}
\end{proof}

Let now $N$ be a split Nakamura manifold, say $N = \Gamma \backslash G$ with $\Gamma = \Gamma' \ltimes_\rho \Lambda_\tau$. As the projection to the last factor $G = \IC^n \rtimes_\rho \IC \longrightarrow \IC$ is equivariant with respect to the action by left translations by elements of $\Gamma$ on $G$ and translations by elements of $\Lambda_\tau$ on $\IC$ there is a well defined holomorphic map on the quotients:
\[\pi: N \longrightarrow B = \IC / \Lambda_\tau,\]
where $B$ is a complex $1$-dimensional torus. We can then compute by Proposition \ref{prop: iso cohomology} that
\[\begin{array}{rl}
h^{1, 0}_{\delbar}(N) = & 1 + \#\set{i \st \lambda_i \frac{\Re(\tau)}{\Im(\tau)}, \lambda_i \frac{|\tau|^2}{\Im(\tau)} \in 2\pi\IZ} =\\
= &  1 + \#\set{i \st \lambda_i = 0} + \#\set{i \st \lambda_i \neq 0, \lambda_i \frac{\Re(\tau)}{\Im(\tau)}, \lambda_i \frac{|\tau|^2}{\Im(\tau)} \in 2\pi\IZ}.
\end{array}\]

As a consequence, if $h^{1, 0}_{\delbar}(N) = 1$ (which happens, e.g., if $N$ satisfies the $\del\delbar$-Lemma and $\lambda_i \neq 0$ for every $i$) then by Proposition \ref{prop: albanese general} we conclude that $\pi: N \longrightarrow B$ is the Albanese morphism of $N$.

\section{Automorphisms of split Nakamura manifolds}\label{sect: automorphisms}
In this Section we want to describe automorphisms of split Nakamura manifolds. Identify $\pi_1(N) = \Gamma$ with the group of deck transformations of the universal cover $G$ of $N$ via
\[\Gamma \simeq \set{L_{(\beta, \alpha)}: G \longrightarrow G \st (\beta, \alpha) \in \Gamma},\]
then we have a canonical identification
\[\Aut(N) \simeq \frac{N_{\Aut(\IC^2)}(\Gamma)}{\Gamma},\]
where
\[N_{\Aut(\IC^2)}(\Gamma) = \set{F \in \Aut(\IC^2) \st F \circ \Gamma \circ F^{-1} = \Gamma}\]
is the normalizer of $\Gamma$ in $\Aut(\IC^2)$.

Let $f \in \Aut(N)$ be any automorphism; as the quotient $\IC^{n + 1} \longrightarrow N$ is the universal cover of $N$, it is naturally induced an automorphism $F \in \Aut(\IC^{n + 1})$ such that the diagram
\[\xymatrix{\IC^{n + 1} \ar[r]^F \ar[d] & \IC^{n + 1} \ar[d]\\
N \ar[r]^f & N}\]
commutes. Moreover, the lift $F$ normalizes $\Gamma$, hence for every $(\beta, \alpha) \in \Gamma$ there exists $(\beta', \alpha') \in \Gamma$ such that
\begin{equation}\label{eq: F normalizes}
F \circ L_{(\beta, \alpha)} = L_{(\alpha', \beta')} \circ F.
\end{equation}

Write $F = \left( \begin{array}{c} H\\ \hline h \end{array} \right)$, with $H: \IC^{n + 1} \longrightarrow \IC^n$ and $h: \IC^{n + 1} \rightarrow \IC$ holomorphic, then the differential of \eqref{eq: F normalizes} is
\begin{equation}\label{eq: diff normalizing}
\left( \begin{array}{c|c}
\frac{\partial H}{\partial z} & \frac{\partial H}{\partial w}\\
\hline
\frac{\partial h}{\partial z} & \frac{\partial h}{\partial w}
\end{array} \right)_{|_{L_{(\beta, \alpha)}(z, w)}} \cdot \left( \begin{array}{c|c}
\rho(\alpha) & 0\\
\hline
0 & 1
\end{array} \right) = \left( \begin{array}{c|c}
\rho(\alpha') & 0\\
\hline
0 & 1
\end{array} \right) \cdot \left( \begin{array}{c|c}
\frac{\partial H}{\partial z} & \frac{\partial H}{\partial w}\\
\hline
\frac{\partial h}{\partial z} & \frac{\partial h}{\partial w}
\end{array} \right).
\end{equation}
In order to have information on $F$, we analyze the entries of this matrix separately.

As $\frac{\partial h}{\partial w}(L_{(\beta, \alpha)}(z, w)) = \frac{\partial h}{\partial w}(z, w)$, we deduce that $\frac{\partial h}{\partial w}$ descends to a function on $N$. Since $N$ is compact, this function must be constant and so
\[h(z, w) = h_1(z) + tw\]
for suitable $t \in \IC$ and $h_1: \IC^{n} \longrightarrow \IC$ holomorphic.

From \eqref{eq: diff normalizing} again, we therefore deduce that
\begin{equation}\label{eq: diff h_1}
\frac{\partial h_1}{\partial z}(\beta + \rho_\tau(\alpha) \cdot z) \cdot \rho_\tau(\alpha) = \frac{\partial h_1}{\partial z}(z).
\end{equation}
Since this must hold for every $(\beta, \alpha) \in \Gamma$, in particular it must be true for the elements of the form $(\beta, 0) \in \Gamma$, for which $\rho(0) = \id$. Hence $\frac{\partial h_1}{\partial z}$ descends to a function on the complex torus $\IC^n / \Gamma'$ and so $\frac{\partial h_1}{\partial z} = (s_1, \ldots, s_n)$ is a constant row vector. As a consequence, \eqref{eq: diff h_1} reduces to
\[(s_1, \ldots, s_n) \cdot \rho(\alpha) = (s_1, \ldots, s_n) \qquad \forall \alpha \in \Lambda_\tau,\]
i.e., $(s_1, \ldots, s_n)^T \in \bigcap_{\alpha \in \Lambda_\tau} \Fix(\rho(\alpha))$. Thank to the description of these fixed locus in Section \ref{sect: faithfulness and freedom} we easily deduce the following result.

\begin{prop}
Assume that $\lambda_i \neq 0$ for every $i = 1, \ldots, n$. Then
\[h(z, w) = tw + \sigma\]
for suitable $t, \sigma \in \IC$ such that $t \cdot \Lambda_\tau \subseteq \Lambda_\tau$.
\end{prop}
\begin{proof}
The discussion in this Section allows us to conclude that $h$ has the prescribed form, so we only have to show that $t \cdot \Lambda_\tau \subseteq \Lambda_\tau$. The last component of the evaluation of \eqref{eq: F normalizes} in $(z, w)$ is
\[t(\alpha + w) + \sigma = \alpha' + tw + \sigma,\]
so we deduce that
\begin{equation}\label{eq: a'}
\alpha' = t \cdot \alpha \in \Lambda_\tau
\end{equation}
as desired.
\end{proof}

Assume from now on that $\lambda_i \neq 0$ for every $i = 1, \ldots, n$.

We deduce from \eqref{eq: diff normalizing} that
\[\frac{\partial H}{\partial w}(L_{(\beta, \alpha)}(z, w)) = \rho(\alpha') \cdot \frac{\partial H}{\partial w}(z, w),\]
and, in particular, this implies that for the elements of the form $(\beta, 0) \in \Gamma$ that
\[\frac{\partial H}{\partial w}(\beta + z, w) = \frac{\partial H}{\partial w}(z, w).\]
In fact, it follows from \eqref{eq: a'} that the element $(\beta', \alpha') \in \Gamma$ satisfying \eqref{eq: F normalizes} must have $\alpha' = 0$ whenever $\alpha = 0$. From this last equality we deduce that $\frac{\partial H}{\partial w}$ does not depend on $z$, but just on $w$. As a consequence
\[H(z, w) = H_1(z) +H_2(w).\]

With this information, we deduce from \eqref{eq: diff normalizing} that then
\[\frac{\partial H_1}{\partial z}(\beta + \rho(\alpha)z) \cdot \rho(\alpha) = \rho(\alpha') \cdot \frac{\partial H_1}{\partial z}(z),\]
so by the same argument as before we deduce that $\frac{\partial H_1}{\partial z} = A$ is constant. Hence $H$ is of the form
\[H(z, w) = Az + H_2(w)\]
for a suitable $A \in \operatorname{Mat}(n, n; \IC)$ and $H_2: \IC \longrightarrow \IC^n$ holomorphic.

\begin{prop}\label{prop: form of lift}
Let $f \in \Aut(N)$, where $N$ is a split Nakamura manifold with $\lambda_i \neq 0$ for every $i = 1, \ldots, n$. Let $F: \IC^{n + 1} \longrightarrow \IC^{n + 1}$ be a lift of $f$ to the universal cover. Then $F$ is of the form
\[F(z, w) = (Az + H_2(w), tz + \sigma)\]
for suitable $t, \sigma \in \IC$, $A \in \operatorname{Mat}(n, n; \IC)$ and $H_2: \IC \longrightarrow \IC^n$ holomorphic such that $t \cdot \Lambda_\tau \subseteq \Lambda_\tau$, $A \cdot \Gamma' \subseteq \Gamma'$ and
\begin{equation}\label{eq: relations}
\rho(t\alpha) \cdot A = A \cdot \rho(\alpha), \qquad \rho(t\alpha) \cdot \frac{\partial H_2}{\partial w}(w) = \frac{\partial H_2}{\partial w}(\alpha + w).
\end{equation}
\end{prop}
\begin{proof}
The fact that $f$ lifts to a map of the stated form follows from the discussion in this section.

We show that $A$ preserves $\Gamma'$. Recall that for $(\beta, \alpha)$ of the form $(\beta, 0)$ the corresponding $(\beta', \alpha')$ satisfying \eqref{eq: F normalizes} is still of the form $(\beta', 0)$. This implies that
\[\begin{array}{rl}
(A(\beta + z) + H_2(w), tw + \sigma) = & F(L_{(\beta, 0)}(z, w)) =\\
= & L_{(\beta', 0)}(F(z, w)) =\\
= & (\beta' + Az + H_2(w), tw + \sigma)
\end{array}\]
and so we deduce that $A \cdot \Gamma' \subseteq \Gamma'$.

The two equalities in \eqref{eq: relations} are the specialization of \eqref{eq: diff normalizing} when $F$ is of the form $F(z, w) = (Az + H_2(w), tw + \sigma)$.
\end{proof}

\begin{prop}
Let $F(z, w) = (Az + H_2(w), tw + \sigma)$ be the lift to $\IC^{n + 1}$ of an automorphism $f \in \Aut(N)$, as in Proposition \ref{prop: form of lift}. Then $A$ and $t$ are invertible.
\end{prop}
\begin{proof}
Let $F'(z, w) = (A'z + H'_2(z), t'w + \sigma')$ be a lift for $f^{-1}$. As $F' \circ F$ lifts $\id_N$ we have that $F' \circ F \in \Gamma$, namely, there exist $(\beta, \alpha) \in \Gamma$ such that $F' \circ F = L_{(\beta, \alpha)}$. This means that
\[(A'Az + A' \cdot H_2(w) + H_2'(tw + \sigma), t't w + t'\sigma + \sigma') = (\beta + \rho(\alpha) \cdot z, \alpha + w),\]
from which we deduce that
\begin{enumerate}
\item $t't = 1$, hence that $t$ is invertible;
\item $A'A = \rho_\tau(\alpha) = \Delta^{D_\tau(\alpha)}$; hence $\det(A'A) = \det(\Delta^{D_\tau(\alpha)}) = 1$ and $A$ is invertible.
\end{enumerate}
\end{proof}

\begin{cor}
Let $F$ be as in Proposition \ref{prop: form of lift}. Then
\[t \in \set{1, -1} \qquad \text{and} \qquad \Delta^t A = A \Delta.\]
\end{cor}
\begin{proof}
Consider the first equation in \eqref{eq: relations} for $\alpha = \tau$: in this case $\rho(\tau) = \id$ and since $A$ is invertible we deduce that $\rho(t\tau) = \id$. Hence $t\tau \in \ker(\rho) = \Span_{\IR}\set{\tau}$ and we deduce that $t \in \IR$. As a consequence
\[t \in \Lambda_\tau \cap \IR = \IZ\]
and so $t \in \set{1, -1}$ as it is invertible.

Finally, the first equation in \eqref{eq: relations} for $\alpha = 1$ is $\rho(t) A = A \rho(1)$ which is the second equality in the statement.
\end{proof}

We can then state the main theorem of this section.

\begin{thm}
Let $N$ be a split Nakamura manifold, corresponding to parameters $\lambda_1, \ldots, \lambda_n \in \IR$ and $\tau \in \cH$. Assume that none of the $\lambda_i$ vanishes. Let $f \in \Aut(N)$ be any automorphism and lift it to $F: \IC^{n + 1} \longrightarrow \IC^{n + 1}$. Then $F$ is of the form
\[F(z, w) = (Az + E(w) + h, tw + \sigma)\]
where $t \in \set{1, -1}$, $A \in \GL(n, \IC)$, $h \in \IC^n$ and $\sigma \in \IC$ satisfy $\Delta^t A = A \Delta$, $A \cdot \Gamma' = \Gamma'$ and $(\id - \Delta)h \in \Gamma'$ and, finally, $E: \IC \longrightarrow \IC^n$ is a holomorphic function such that $\rho_\tau(t \alpha) \cdot E(w) = E(w + \alpha)$ for every $\alpha \in \Lambda_\tau$ (which implies that $E(w) \equiv 0$ if $\frac{\Re(\tau)}{|\tau|^2} \notin \IQ$).\\
Vice versa, every function $F: \IC^{n + 1} \longrightarrow \IC^{n + 1}$ as above descends to an automorphism of $N$.\\
Moreover, we have $F \circ L_{(\beta, \alpha)} = L_{(\beta', \alpha')} \circ F$ for every $(\beta, \alpha) \in \Gamma$ with
\[(\beta', \alpha') = (A \beta + (\id_{\IC^n} - \rho_\tau(t\alpha))h), t \alpha).\]
\end{thm}
\begin{proof}
Let
\[f(z, w) = (Az + H_2(w), tw + \sigma)\]
be a lift of $f$ to $\IC^{n + 1}$ as in Proposition \ref{prop: form of lift}, and write $H_{2, i}(w)$ for the $i^\text{th}$ component of $H_2$. By the second equation in \eqref{eq: relations}, $H_{2, i}(w)$ satisfies
\[e^{\lambda_i D(\alpha)} \frac{\partial H_{2, i}}{\partial w}(w) = \frac{\partial H_{2, i}}{\partial w}(\alpha + w),\]
so by Corollary \ref{cor: f = 0} (see Section \ref{sect: technical lemma}) we have
\[H_{2, i}(w) = \underbrace{\frac{c_i}{2\pi\ii m_i + t \lambda_i} e^{(2\pi\ii m_i + t \lambda_i)w}}_{E_i(w)} + h_i, \qquad c_i, h_i \in \IC, m_i \in \IZ,\]
with $c_i = 0$ if it is not possible to write $\tau = \tau(t \lambda_i, m_i, k_i)$ for any $m_i, k_i \in \IZ$ (in case $\frac{\Re(\tau)}{|\tau|^2} \notin \IQ$ we have $c_i = 0$ for every $i$). Observe that the function $E = (E_1, \ldots, E_n)^T$ satisfies
\[\rho(t\alpha) \cdot E(w) = E(\alpha + w), \qquad \text{for every } \alpha \in \Lambda_\tau.\]

It follows from \eqref{eq: F normalizes} that $\beta' + \rho(t\alpha) h = A\beta + h$, i.e., that
\[\pa{\id - \Delta^{t D_\tau(\alpha)}}h = \beta' - A\beta \in \Gamma'.\]
Using the fact that $\Delta$ preserves $\Gamma'$, it is not hard to prove by induction that $\pa{\id - \Delta^k}h \in \Gamma'$ for every $k \in \IZ$ if and only if $(\id - \Delta)h \in \Gamma'$.

We have proven up to now that the parameters $t, A, h, \sigma$ of the lift satisfy $t \in \set{1, -1}$, $\Delta^t A = A \Delta$, $A \cdot \Gamma' \leq \Gamma'$ and $(\id - \Delta)h \in \Gamma'$. Observe that these conditions are equivalent to the fact that $F \circ \Gamma \circ F^{-1} \subseteq \Gamma$. Consider then the inverse map
\[F^{-1}(z, w) = \pa{A^{-1}z - A^{-1}h, tw - t\sigma},\]
as $F^{-1}$ lifts $f^{-1}$ the same conclusions must hold for $F^{-1}$. In particular, it holds that:
\begin{enumerate}
\item $t \in \set{1, -1}$ which is already the case;
\item $\Delta^t A^{-1} = A^{-1} \Delta$, which follows from $\Delta^t A = A \Delta$ for the particular case where $t \in \set{1, -1}$;
\item $A^{-1} \cdot \Gamma' \leq \Gamma'$;
\item $-(\id - \Delta)A^{-1}h \in \Gamma'$, which is true since $(\id - \Delta)A^{-1}h = A^{-1}(\id - \Delta^t)h$, $A^{-1}$ preserves $\Gamma'$ and $(\id - \Delta^t)h \in \Gamma'$ as $(\id - \Delta)h \in \Gamma'$.
\end{enumerate}
Hence we deduce that $A \cdot \Gamma' = \Gamma'$ and since these last conditions express the fact that $F^{-1} \circ \Gamma \circ F \subseteq \Gamma$ we conclude that in fact $F$ normalizes $\Gamma$, and the theorem is proven.
\end{proof}

\subsection{A technical Lemma}\label{sect: technical lemma}

As it is clear from the work in Section \ref{sect: automorphisms}, we are interested in study the holomorphic functions $f: \IC \longrightarrow \IC$ such that
\[e^{\lambda D_\tau(\alpha)} f(w) = f(\alpha + w) \qquad \text{for every } \alpha \in \Lambda_\tau.\]
In fact, this is the form in which appears every component of the second equation in \eqref{eq: relations}.

\begin{lemma}
Let $\lambda \in \IR \smallsetminus \set{0}$ and $\tau \in \cH$. Then the holomorphic functions $f: \IC \longrightarrow \IC$ such that $f(z + 1) = e^\lambda f(w)$ and $f(w + \tau) = f(w)$ for every $w \in \IC$ are those of the form
\[f(w) = c \cdot e^{(2\pi\ii m + \lambda)w}, \qquad c \in \IC.\]
If, moreover, $\frac{\Re(\tau)}{|\tau|^2} \notin \IQ$, then $c = 0$, i.e., $f$ is the trivial function $f(w) \equiv 0$.
\end{lemma}
\begin{proof}
The condition that $f(w + 1) = e^\lambda f(w)$ is equivalent to require that the holomorphic function $g(w) = \frac{f(w)}{e^{\lambda w}}$ is $1$-periodic. Such a $g$ can be written in Fourier series as
\[g(w) = \sum_{m \in \IZ} B_m e^{2\pi \ii mw}\]
for suitable $B_m \in \IC$, so we have that $f$ has also the Fourier series expansion
\[f(w) = e^{\lambda w} g(w) = \sum_{m \in \IZ} B_m e^{(2\pi\ii m + \lambda)w}.\]
It remains to control when $f$ is $\tau$-periodic. Assume that $m_0 \in \IZ$ is such that $B_{m_0} \neq 0$, then we must have $e^{(2\pi\ii m_0 + \lambda)(w + \tau)} = e^{(2\pi\ii m_0 + \lambda)w}$, i.e., $(2\pi\ii m_0 + \lambda) \tau = 2k\pi\ii \in 2\pi\ii \cdot \IZ$ for a suitable $k \in \IZ$. Hence
\[\tau = \frac{2k\pi}{\lambda^2 + 4\pi^2 m_0^2}(2\pi m_0 + \ii \lambda) = \tau(\lambda, m_0, k),\]
which satisfies $\frac{\Re(\tau)}{|\tau|^2} = \frac{m_0}{k} \in \IQ$. As $\lambda$ is fixed, by Lemma \ref{lemma: preimages of tau} we deduce that $(\lambda, m_0, k)$ is uniquely determined, hence that $B_{m_0}$ is the only non trivial coefficient.

This shows that
\[f(w) = c \cdot e^{(2\pi\ii m + \lambda)w}, \qquad c \in \IC\]
and if $\frac{\Re(\tau)}{|\tau|^2} \notin \IQ$ then $c = 0$.
\end{proof}

\begin{cor}\label{cor: f = 0}
Let $f: \IC \longrightarrow \IC$ be a holomorphic function such that
\[e^{\lambda D_\tau(\alpha)} f(w) = f(\alpha + w) \qquad \text{for every } \alpha \in \Lambda_\tau\]
for some $\lambda \in \IR \smallsetminus \set{0}$ and $\tau \in \cH$. Then
\[f(w) = c \cdot e^{(2\pi\ii m + \lambda)w}, \qquad c \in \IC,\]
with $c = 0$ if $\frac{\Re(\tau)}{|\tau|^2} \notin \IQ$.
\end{cor}
\begin{proof}
Just observe that $f$ satisfies the relation in the statement if and only if $e^\lambda f(w) = f(w + 1)$ and $f(w) = f(w + \tau)$ for every $w \in \IC$.
\end{proof}

\bibliographystyle{alpha}
\bibliography{NakamuraMfldsBib}

\end{document}